\documentclass[11pt]{article}
\usepackage[top=2.5cm,bottom=2.5cm,left=2.5cm,right=2.5cm]{geometry}
\usepackage{mathrsfs}
\usepackage{pstricks}
\usepackage{amsmath,amssymb,graphicx,amsthm,tabularx,array}
\usepackage{hyperref}
\usepackage{makecell}

\theoremstyle{plain}
\newtheorem{theorem}{Theorem}[section]

\newtheorem{coro}[theorem]{Corollary}
\newtheorem{lem}[theorem]{Lemma}

\newtheorem{fact}[theorem]{Fact}
 
\newtheorem{conj}[theorem]{Conjecture}

\theoremstyle{definition}

\newtheorem{other}{}

\title{The signless Laplacian spectral Tur\'{a}n problems for color-critical graphs}

\author{
Jian Zheng\thanks{
School of  Mathematics and Statistics, Jiangxi Normal University, Jiangxi, China. 
Email: \url{zhengj@jxnu.edu.cn}. } 
 \and 
 Yongtao Li\thanks{School of Mathematics and Statistics, Central South University, Changsha, China. Email: \url{ytli0921@hnu.edu.cn}. 
Supported by the Postdoctoral Fellowship Program of CPSF (No. GZC20233196).}
 \and 
 Honghai Li\thanks{Corresponding author, 
 School of  Mathematics and Statistics, Jiangxi Normal University, Jiangxi, China. 
Email: \url{lhh@jxnu.edu.cn}.  
Supported by National Natural Science Foundation of China (No. 12161047) and  Jiangxi Provincial Natural Science foundation (No. 20224BCD41001).} 
}

\date{\today}
\begin{document}
\maketitle
\begin{abstract}
The well-known Tur\'{a}n theorem states that if $G$ is an $n$-vertex $K_{r+1}$-free graph, then $e(G)\le e(T_{n,r})$, with equality if and only if $G$ is the $r$-partite Tur\'{a}n graph $T_{n,r}$. 
A graph $F$ is called color-critical if it contains an edge whose deletion reduces its chromatic number. 
Extending the Tur\'{a}n theorem, Simonovits (1968) proved that for any color-critical graph $F$ with $\chi (F)=r+1$ and sufficiently large $n$, the Tur\'{a}n graph $T_{n,r}$ is the unique graph with maximum number of edges among all $n$-vertex $F$-free graphs. 
Subsequently, Nikiforov [Electron. J. Combin., 16 (1) (2009)] proved a spectral version of the Simonovits theorem in terms of the adjacency spectral radius. 
In this paper, we show an extension of the Simonovits theorem for the signless Laplacian spectral radius.  We prove that for any color-critical graph $F$ with $\chi (F)=r+1\ge 4$ and sufficiently large $n$, if $G$ is an $F$-free graph on $n$ vertices, then $q(G)\le q(T_{n,r})$, with equality if and only if $G=T_{n,r}$.  
Our approach is to establish a signless Laplacian spectral version of the criterion of Keevash, Lenz and Mubayi [SIAM J. Discrete Math., 28 (4) (2014)]. 
Consequently, we can determine the signless Laplacian spectral extremal graphs for generalized books and even wheels.  
As an application, our result gives an upper bound on the degree power of an $F$-free graph. 
We show that if $n$ is sufficiently large and $G$ is an $F$-free graph on $n$ vertices with $m$ edges, then $\sum_{v\in V(G)} d^2(v) \le 2(1- \frac{1}{r})mn$, with equality if and only if $G$ is a regular Tur\'{a}n graph $T_{n,r}$.  
This extends a result of Nikiforov and Rousseau [J. Combin. Theory Ser B 92 (2004)].  
\end{abstract}

 {\bf MSC classification}\,: 15A42, 05C50
 
 {\bf Keywords}\,: Extremal graph theory; 
 Tur\'{a}n theorem; 
  Color-critical graphs; 
 Signless Laplacian spectral radius;  
 Degree powers.

\section{Introduction}

A graph $G=(V,E)$ consists of a vertex set $V=\{v_1,v_2,\ldots,v_n\}$  and an edge set $E=\{e_1,e_2,\ldots,e_m\}$, where each edge $e_i$ is a $2$-element subset of $V$.
 We write $|G|:=|V|=n$ and $e(G):= |E|=m$ for the number of vertices and edges in $G$, respectively. 
 For a graph $G$ and a vertex $u\in V(G)$, the  {degree} $d_{G}(u)$ of $u$ is the number of edges in $G$ containing $u$. 
 We write $N (u) $ for the set of neighbors of $u$. 
 Let $\delta(G)$ be the {minimum degree} of $G$. Let $G-u$ be the subgraph obtained from $G$ by removing $u$ from $V(G)$ and removing all edges containing $u$ from $E(G)$. 
 Let $K_{r+1}$ be the complete graph on $r+1$ vertices, and $K_{s,t}$ be the complete bipartite graph with partite sets of sizes $s$ and $t$. 
 For two vertex-disjoint graphs $G$ and $H$, 
 the join $G\vee H$ is obtained from the union $G\cup H$ by joining each vertex of $G$ to each vertex of $H$. 
A {\it proper vertex coloring} of a graph $F$ is an  assignment of colors to the vertices so that the adjacent vertices get different colors. The {\it chromatic number} of $F$, denoted by $\chi(F)$, is the minimum number of colors of a proper vertex coloring of $F$.

\subsection{The classical extremal graph results} 

Let $G$ and $F$ be two simple graphs.  We say that $G$ is {\it $F$-free} if $G$ does not contain a subgraph isomorphic to $F$. Clearly, every bipartite graph is triangle-free. 
A classic problem in extremal graph theory is the Tur\'{a}n-type problem, 
which asks for the maximum number of edges in an $F$-free graph of order $n$.  
The investigation of Tur\'{a}n-type problems can be traced back to Mantel's seminal theorem in 1907, which established the foundational result in extremal graph theory; see \cite{Bollobas78}.  
Mantel's theorem states that every triangle-free graph $G$ of order $n$ contains at most $\lfloor {n^2}/{4}\rfloor$ edges,  
and the maximum is achieved if and only if $G$ is a balanced complete bipartite graph.   
Extending Mantel's theorem, 
Tur\'{a}n \cite{Turan41} proved that if $G$ is an $n$-vertex $K_{r+1}$-free graph, then 
\begin{equation} \label{eq-weak}
  e(G)\le \left( 1- \frac{1}{r}\right) \frac{n^2}{2}. 
  \end{equation} 
 Let $T_{n,r}$ be the $n$-vertex $r$-partite {\it Tur\'{a}n graph}, which is the complete $r$-partite graph on $n$ vertices whose partite sets are as balanced as possible. 
 Clearly, we can see that $T_{n,r}$ is $K_{r+1}$-free and 
 $e(T_{n,r}) \le (1- \frac{1}{r}) \frac{n^2}{2}$. 
Moreover, it is easy to check that if $r$ divides $n$, then $e(T_{n,r}) = {r \choose 2} (\frac{n}{r})^2 = (1-\frac{1}{r})\frac{n^2}{2}$. 
 In fact, Tur\'{a}n \cite{Turan41} proved a  slightly stronger version than (\ref{eq-weak}).

\begin{theorem}[Tur\'{a}n \cite{Turan41}]  \label{thm-Turan}
If  $G$ is a $K_{r+1}$-free graph on $n$ vertices, 
then 
\[  e(G)\le e(T_{n,r}), \] 
where the equality holds if and only if $G=T_{n,r}$.
\end{theorem}

Many different proofs of Tur\'{a}n's theorem could be found in the literature;
see \cite[pp. 294--301]{Bollobas78} for more details.
Furthermore,
there are various extensions and generalizations on Tur\'{a}n's theorem;
see, e.g., \cite{BT1981,Bon1983}. 
We say that a graph $F$ is {\it color-critical} if   there exists an edge $e$ of $F$ such that $\chi(F-e)<\chi(F)$, where $F-e$ denotes the graph obtained from $F$ by deleting the edge $e$. 
This is a very broad and important class of graphs. 
For example, cliques, odd cycles, wheels with even order, cliques with one edge removed, complete bipartite graphs plus an edge, books, joints, and the Grotzsch graph are color-critical; see \cite{Mub2010,PY2017}. 
Extending the Tur\'{a}n theorem, Simonovits \cite{S1968} proved the following result by using the so-called `progressive induction'.  

\begin{theorem}[Simonovits \cite{S1968}] 
\label{thm-Sim}
Let $F$ be a color-critical graph with $\chi(F)=r+1\geq 3$. Then there exists $n_{0}$ such that for any $F$-free graph $G$ on $n\geq n_{0}$ vertices, we have 
\[  e(G)\leq e(T_{n,r}), \] 
 where the equality holds if and only if $G=T_{n,r}$.
\end{theorem}

 In the last few decades, there has been significant development in the field of spectral graph theory, which is an advanced and specialized area of combinatorics that combines graph theory with algebraic theory; see \cite{BH2012,GR2001}.
Another recent trend in extremal graph problems motivated by Turán's theorem is significant progress in spectral extremal graph theory, which studies the extremal structures of graphs using the eigenvalues of various matrices associated with graphs, such as the adjacency matrix and the signless Laplacian matrix. Many classical edge-extremal results have been translated into spectral statements, which not only lead to improvements but also deepen the understanding of the interplay between edge-extremal and spectral-extremal problems; see \cite{Niki2011}.

The {\it adjacency matrix} of a graph $G$ is defined as $A(G)=[a_{ij}]_{n\times n}$ with $a_{ij}=1$ if $ij\in E(G)$, and $a_{ij}=0$ otherwise. The {\it adjacency spectral radius} of $G$, denoted by $\lambda (G)$, is the maximum modulus of the eigenvalues of $A(G)$. 
In 1986, Wilf \cite{Wil1986} proved that if $G$ is a $K_{r+1}$-free graph on $n$ vertices, then
\begin{equation} \label{eq-Wilf}
    \lambda (G) \le \left( 1- \frac{1}{r}\right)n.
\end{equation}
This implies the Tur\'{a}n bound in (\ref{eq-weak}) immediately by  the fundamental inequality $\lambda (G)\ge \frac{2e(G)}{n}$.
Nikiforov \cite{Niki2007laa2} and
Guiduli \cite{Gui1996} independently obtained the spectral version of Theorem \ref{thm-Turan} in terms of the adjacency spectral radius. 
Furthermore,  Nikiforov \cite{N2009} generalized this result and proved the adjacency spectral analog of Theorem \ref{thm-Sim}. Later, Nikiforov proposed the spectral Tur\'{a}n-type problem: What is the maximum spectral radius of an $n$-vertex graph without certain subgraphs? 
 This spectral problem has rapidly emerged as a focal point of research in extremal graph theory  and has yielded a lot of significant results. We refer the reader to related surveys \cite{Niki2011,LiLF}.

\subsection{The signless Laplacian spectral radius}

In this paper, we shall pay attention mainly to another significant spectral extension of the Tur\'{a}n theorem and the Simonovits theorem. 
The {\it signless Laplacian matrix} of a graph $G$ is defined as 
\[ Q(G):=D(G)+A(G),\]
where $D(G)$ and $A(G)$ are the degree diagonal matrix and the adjacency matrix of $G$, respectively. 
The {\it signless Laplacian spectral radius} (also called the {\it $Q$-index}) of $G$, denoted by  $q(G)$, is the maximum modulus of eigenvalues of $Q(G)$. It is a well-known fact that 
\begin{equation} \label{eq-bound-q}
  \frac{4e(G)}{n} \le 2 \lambda (G) \le  q(G)\le 2\Delta (G), 
\end{equation}
where $\Delta (G)$ is the maximum degree of $G$. 
 So any upper bound on $q(G)$ can yield corresponding bounds on $\lambda (G)$ and $e(G)$, respectively. 
It was shown by de Abreu and Nikiforov \cite{AN2013} that if $G$ is an $n$-vertex $K_{r+1}$-free graph, then  
\begin{equation} \label{eq-AN}
    q(G) \le \left( 1-\frac{1}{r}\right)2n. 
\end{equation} 
This result extends both (\ref{eq-weak}) and (\ref{eq-Wilf}) due to (\ref{eq-bound-q}). 
 He, Jin, and Zhang \cite[Theorem $1.3$]{HJZ2013} sharpened (\ref{eq-AN}) and proved an extension of Theorem \ref{thm-Turan} in terms of the signless Laplacian spectral radius.

\begin{theorem}[He--Jin--Zhang \cite{HJZ2013}]\label{partite}
If $G$ is an $n$-vertex $K_{r+1}$-free graph, then 
\[  q(G)\leq q(T_{n,r}).\] 
Moreover, the equality holds if and only if $G$ is a complete bipartite graph for $r=2$ and the $r$-partite Tur\'{a}n graph $T_{n,r}$ for every $r\geq3$.
\end{theorem} 

Note that when $r=2$,  the extremal graph achieving the maximum signless Laplacian spectral radius is not unique, which
is slightly different from the case of the adjacency spectral radius. 

In $2013$, Freitas, Nikiforov and Patuzzi \cite{FNP2013} proposed the Tur\'an-type extremal problem for the signless Laplacian spectral radius: 
 What is the maximum signless Laplacian spectral radius of an $F$-free graph of order $n$? 
Many significant results for some specific graphs have already been investigated in the past few decades; see, e.g.,  paths \cite{NY1}, cycles \cite{FNP2013,NY2,Y2014,CWZ2022}, 
complete bipartite graphs \cite{AFNP2016,CZ2021}, 
Hamilton cycles \cite{Zhou2010,LLPLMA18}, 
linear forests \cite{CLZ2020}, 
matchings \cite{Yu2008,PLZ2020}, 
friendship graphs \cite{ZHG2021}, 
cliques \cite{ZXL2020,LMX2022,ZW2024}, 
flowers \cite{CLZ2024}, fan graphs \cite{WZ2023} and 
books \cite{KW2017,CJZ2025}. 

\medskip 
\noindent 
{\bf Motivation.} 
In the references mentioned above, the majority of conclusions primarily focus on graphs $F$ with chromatic number $\chi(F) \le 3$.  
Let $S_{n,k} := K_k \vee I_{n-k}$ denote the split graph obtained by joining every vertex of the complete graph $K_k$ to every vertex of an independent set of size $n-k$. 
The graph $S_{n,k}$ plays an important role in the study of the signless Laplacian spectral radius, comparable to the role of the Tur\'{a}n graph $T_{n,r}$ in classical Tur\'{a}n-type problems.   
Apart from cliques, there is no existing research on graphs $F$ with $\chi(F) \ge 4$.  
This phenomenon suggests that studying the signless Laplacian spectral Tur\'{a}n-type problem for graphs $F$ with large chromatic number may be substantially more challenging. 
To elaborate, most references concentrate on the problems for $F$-free graphs with $\chi(F) \le 3$, because current methods rely heavily on an upper bound on the signless Laplacian spectral radius involving the degrees of vertices  \cite{Mer1998}: 
\[  q(G) \le \max_{v \in V(G)} \left\{ d(v) + \frac{1}{d(v)} \sum_{w \in N(v)} d(w) \right\}. \]    
This bound is effective for spectral graph problems where the extremal graph is structurally close to a split graph $S_{n,k}$, as it helps characterize extremal graph structures through vertex degrees. However, for extremal graphs resembling Tur\'{a}n graphs $T_{n,r}$, this upper bound is less useful. Consequently, new methods are required to address the $F$-free graphs with $\chi(F) \ge 4$.

\section{Main results}

In this paper, we determine the maximum signless Laplacian spectral radius of an $n$-vertex graph excluding any color-critical graph $F$ with chromatic number $\chi(F) \ge 4$. 
Our result not only extends Theorem \ref{partite}, but can also be viewed as a spectral extension of Theorem \ref{thm-Sim}.

\begin{theorem}[Main result] \label{color} 
Let $F$ be a color-critical graph with $\chi(F)=r+1\geq 4$. Then there exists $n_{0}$ such that for every  $F$-free graph $G$ on $n\geq n_{0}$ vertices, we have 
\[  q(G)\leq q(T_{n,r}), \] 
 where the equality holds if and only if $G$ is the $r$-partite Tur\'{a}n graph $T_{n,r}$.
\end{theorem}

\noindent 
{\bf Remark.} The assumption $\chi (F) \geq 4$ in Theorem \ref{color} is required. 
This is slightly different from the assumption in Theorem \ref{thm-Sim}. For example, we choose $F$ as an odd cycle $C_{2k+1}$ for every $k\ge 2$. Observe that $C_{2k+1}$ is color-critical and $\chi(C_{2k+1})=3$. 
However, the graph $S_{n,k}$ is $C_{2k+1}$-free and $q(S_{n,k})> n=q(T_{n,2})$. 
In fact, it was shown in \cite{FNP2013,Y2014} that $S_{n,k}$ achieves the maximum signless Laplacian spectral radius among all $n$-vertex $C_{2k+1}$-free graphs for every $n\ge 110k^2$.

\begin{fact} \label{fact}
    If $G$ is an $n$-vertex graph with $q(G)\le q(T_{n,r})$, then $e(G)\le e(T_{n,r})$. 
\end{fact}

 By a direct calculation, one can verify that 
$ \frac{n}{4}q(T_{n,r})< e(T_{n,r}) +1$. If  $q(G)\le q(T_{n,r})$, then  
$e(G)\le \frac{n}{4}q(G)\le \frac{n}{4}q(T_{n,r})< e(T_{n,r}) +1$; see \cite[Corollary 2.5]{HJZ2013} for details. 
Hence, we emphasize that Theorem \ref{color} implies Theorem \ref{thm-Sim} except for the uniqueness of the extremal graph.

The {\it generalized book} $B_{r,k}$ 
is a graph obtained from $k$ copies of 
$K_{r+1}$ sharing a common $K_r$. 
In other words, we have $B_{r,k}:=K_r\vee I_k$, which is obtained by joining every vertex of a clique $K_r$ to every vertex of an independent set $I_k$ of size $k$. 
For $r\ge 3$ and $k\ge 1$, we see that $B_{r,k}$ is color-critical with $\chi (B_{r,k}) = r+1\ge 4$. 
Thus, Theorem \ref{color} 
implies the following result. 

\begin{coro} \label{thm-Brk}
    Let $r\ge 3, k\ge 1$ be fixed and $n$ be sufficiently large. 
If $G$ is a $B_{r,k}$-free graph on $n$ vertices, 
then  
$ q (G)\le q(T_{n,r})$,    
with equality if and only if $G=T_{n,r}$. 
\end{coro}

\noindent 
{\bf Remark.} 
 Corollary \ref{thm-Brk} does not hold in the case $r=2$. Let $G$ be the graph obtained from the complete bipartite graph $K_{n-k+1,k-1}$ by adding an edge to the partite set of size $n-k+1$. Clearly, $G$ is $B_{2,k}$-free and $q(G)> n= q(T_{n,2})$.  
We refer to \cite{CJZ2025} for book-free graphs. 
Extending the Tur\'{a}n theorem, Dirac \cite{Dirac1963} proved that if $n\ge r+2$ and $G$ is an $n$-vertex $B_{r,2}$-free graph, then 
$e(G)\le e(T_{n,r})$. By Fact \ref{fact}, Corollary \ref{thm-Brk} could be viewed as an extension of the result of Dirac.

\medskip 
For an integer $k\ge 2$, 
the {\it even wheel} is defined as $W_{2k+2} :=K_1\vee C_{2k+1}$. 
Observe that $W_{2k+2}$ is color-critical and $\chi (W_{2k+2})=4$. So Theorem \ref{color} yields the following corollary. 

\begin{coro} \label{thm-even-wheel}
For fixed $k\ge 2$ and sufficiently large $n$, 
if $G$ is a $W_{2k+2}$-free  graph on $n$ vertices, then 
$ q(G) \le q(T_{n,3})$, 
with equality if and only if $G=T_{n,3}$. 
\end{coro}

For integers $r\ge 2$ and $k\ge 3$, 
 the {\it generalized wheel} is defined as 
 $W_{r,k}:=K_r \vee C_k$. 
Note that $W_{r,k}$ is color-critical for every $k\ge 3$, since $r\ge 2$ and any deletion of an edge of the clique $K_r$ decreases its chromatic number. 
Clearly, we have $\chi (W_{r,2k+1}) =r+3$ and $\chi (W_{r,2k}) =r+2$.  
Therefore, Theorem \ref{color} implies the following corollary immediately. 

\begin{coro}
    For fixed $r\ge 2$, $k\ge 3$ and sufficiently large $n$, if $G$ is a $W_{r,k}$-free graph on $n$ vertices, then $q(G)\le q(T_{n,r+2})$ for odd $k$, and $q(G)\le q(T_{n,r+1})$ for even $k$. Moreover, the equality holds if and only if $G=T_{n,r+2}$ or $T_{n,r+1}$, respectively. 
\end{coro}

\subsection{Application: Degree powers in graphs}

The second part of this paper is devoted to studying the power of degrees in a graph $G$ that forbids a certain subgraph $F$. 
Recall that $d(v)$ denotes the degree of a vertex $v$. Given a real number $p\ge 1$, {\it the degree power} of a graph $G$ is defined as $\sum_{v\in V(G)} d^p(v)$. This is a well-studied graph parameter and is also known as {\it the $p$-norm} of $G$. 
Observe that in the case $p=1$,
we have $\sum_{v\in V(G)} d(v) =2e(G)$.
In the case $p=2$, the value $\sum_{v\in V(G)} d^2(v)$ is called the first Zagreb index; see, e.g., \cite{LL2009-Zagreb}. It is worth noting that the degree power is closely connected to the generalized Tur\'{a}n problem involving star counts; see \cite{G2025}. 
Motivated by the study of Tur\'{a}n number, there has been a wide investigation on estimating  the degree power of a graph $G$ that does not contain a certain subgraph $F$.  We refer to  \cite{BK2004,BK2012,LLQ2019,Niki2009-degree} and references therein.

 During the study of Ramsey theory,
Nikiforov and Rousseau \cite{NR2004} extended the  classical Tur\'{a}n theorem by showing that if $G$ is a $K_{r+1}$-free graph on
$n$ vertices with $m$ edges, then
\begin{equation}    \label{thm-NR-degree}
\sum_{v\in V(G)} d^2(v) \le 2\left( 1-\frac{1}{r}\right)mn.
\end{equation}
The original proof of (\ref{thm-NR-degree}) uses the combinatorial technique.
Recently,  Li, Liu and Zhang \cite[Sec. 5.4]{LLZ2024-book-4-cycle} provided two algebraic proofs of (\ref{thm-NR-degree}), and they also extended (\ref{thm-NR-degree}) to graphs without the generalized book $B_{r,k}$. 
 Under a similar line of their proofs, we give a unified extension on the sum of squares of degrees in an $F$-free graph $G$ for any color-critical graph $F$.

\begin{theorem} \label{thm-1-6}
For any color-critical graph $F$ with $\chi(F) =r+1\geq 4$,  
there exists $n_0$ such that if $G$ is an $F$-free graph on
$n\ge n_0$ vertices with $m$ edges, then
\[  \sum_{v\in V(G)} d^2(v) \le 2\left( 1-\frac{1}{r} \right)mn,\] 
where the equality holds if and only if 
$G$ is a regular complete $r$-partite graph.  
\end{theorem} 

\noindent
{\bf Remark.}
 Theorem \ref{thm-1-6} implies $m \le (1- \frac{1}{r})\frac{n^2}{2}$.
Indeed, using the Cauchy--Schwarz inequality, we get 
$\sum_{v\in V} d^2(v) \ge \frac{4m^2}{n}$.
So Theorem \ref{thm-1-6}
implies $\frac{4m^2}{n} \le 2\left(1- \frac{1}{r} \right)mn$, that is,
$m\le \left(1- \frac{1}{r} \right) \frac{n^2}{2}$.

Moreover, Theorem \ref{thm-1-6} also implies the following corollary.

\begin{coro}
For any color-critical graph $F$ with $\chi(F) =r+1\geq 4$,  
there exists $n_0$ such that if $G$ is an $F$-free graph on
$n\ge n_0$ vertices, then 
\[  \sum_{v\in V(G)} d^2(v) \le \left( 1-\frac{1}{r} \right)^2 n^3,\]
where the equality holds if and only if 
$G$ is a regular complete  $r$-partite graph.  
\end{coro}

This corollary recovers a result of Bollob\'{a}s and Nikiforov \cite{BK2012} by a quite different method.

\medskip 
\noindent 
{\bf Our approach.} 
Keevash, Lenz and Mubayi \cite{KLM2014} established a criterion, which implies that the spectral extremal problems concerning the $\alpha$-spectral radius of hypergraphs can be derived from the corresponding hypergraph Tur\'an  problems. Motivated by these observations and existing spectral Tur\'an-type results, this paper explores the signless Laplacian spectral Tur\'an-type
problem. Specifically, we present a criterion for the signless Laplacian spectral radius (see Theorem \ref{cri}), which facilitates the transformation of the signless Laplacian spectral Tur\'an-type problem into the classical extremal problem with the so-called degree-stable property (see Lemma \ref{lem-ES1973}).

\section{Some auxiliary lemmas}
In this section, we present a criterion for 
the signless Laplacian spectral radius.  
Our ideas are motivated by the method introduced by  
Keevash, Lenz and Mubayi \cite[Theorem 1.4]{KLM2014}.  
 Although the methods of the proofs are similar, additional finesse is required in some details to resolve the issues for the signless Laplacian spectral radius.

 For a graph  $F$, the {\it Tur\'an number} of $F$, denoted by  $\mathrm{ex}(n,F)$, is defined to be the maximum number of edges of an $F$-free graph on $n$ vertices.   
 The {\it Tur\'an density} of $F$ is defined as $$\pi(F)=\lim\limits_{n\to \infty}\frac{\mathrm{ex}(n,F)}{\binom{n}{2}}.$$ 
 This limit can be shown to exist by a simple monotonicity argument \cite{KNS1964}. 
 The Tur\'{a}n theorem implies that $\pi (K_{r+1})=1- 1/r$. In general, the Erd\H{o}s--Stone--Simonovits 
 theorem  \cite{ES1946,ES1966} states that for every graph $F$ with chromatic number $\chi (F)=r+1$, we have $\pi (F) =1-1/r$.

The main result in this section is as follows:

\begin{theorem}\label{cri}
Let $r\geq 3$ and $F$ be a family of graphs with Tur\'{a}n density $\pi(F) = 1-1/r$. 
For real numbers $0< \varepsilon < 1/2$ and $\sigma < \varepsilon/36$, let $\mathcal{G}_n$ 
be the collection  of all $n$-vertex $F$-free graphs with minimum degree more than $(\pi (F) - \varepsilon)n$. 
We denote $q(\mathcal{G}_{n})=\max\{q(G): G\in \mathcal{G}_{n}\}$. Suppose that there exists $N>0$  such that for every $n\geq N$, we have 
\begin{equation}\label{dl1}
\big|\mathrm{ex}(n,F)-\mathrm{ex}(n-1,F)-\pi(F)n\big|\leq \sigma n,
\end{equation}
and 
\begin{equation}\label{dl2}
\big|q(\mathcal{G}_{n})-4\mathrm{ex}(n,F)n^{-1}\big|\leq \sigma.
\end{equation}
Then there exists $n_{0}\in \mathbb{N}$ such that for any $F$-free graph $H$ on $n\geq n_{0}$ vertices, we have 
$$q(H)\leq q(\mathcal{G}_{n}).$$
In addition, if the equality holds, then $H\in\mathcal{G}_{n}$. 
\end{theorem}

 In the following, we  derive two useful inequalities that are consequence of the assumptions of Theorem \ref{cri}. Suppose that
  $F$, $\sigma$, $n$ and $\mathcal{G}_{n}$ satisfy (\ref{dl1}) and (\ref{dl2}) for all $n\geq N$. 
 Firstly, using (\ref{dl2})
  and the fact that the quotient $\mathrm{ex}(n,F)/\binom{n}{2}$ is decreasing with $n$, and it tends to $\pi(F)$. We have 
  $$\frac{q(\mathcal{G}_{n})}{n}-\frac{4\mathrm{ex}(n,F)}{n^{2}}=o(1).$$
  Therefore, we get 
 \begin{equation}\label{qn}
 q(\mathcal{G}_{n})=(2\pi(F)+o(1))n.
\end{equation}
Secondly, we want to bound the gap between $q(\mathcal{G}_{n})$ and $q(\mathcal{G}_{n-1})$.
 By the triangle inequality, (\ref{dl1}) and (\ref{dl2}), for sufficiently large $n$, we have
 \begin{displaymath}
\begin{split}
& |q(\mathcal{G}_{n})-q(\mathcal{G}_{n-1})-2\pi(F)| \\
&\leq \bigg|\frac{4\mathrm{ex}(n,F)}{n}-\frac{4\mathrm{ex}(n-1,F)}{n-1}-2\pi(F)\bigg|
+2\sigma\\
&= \bigg|\frac{4}{n}\big(\mathrm{ex}(n,F)-\mathrm{ex}(n-1,F)-\pi(F)n\big)+2\pi(F)-\frac{4\mathrm{ex}(n-1,F)}{n(n-1)}\bigg|
+2\sigma\\
&\leq \bigg|2\pi(F)-\frac{4\mathrm{ex}(n-1,F)}{n(n-1)}\bigg|
+6\sigma. 
\end{split}
\end{displaymath}
Then for sufficiently large $n$, we get 
\begin{equation}\label{beg}
|q(\mathcal{G}_{n})-q(\mathcal{G}_{n-1})-2\pi(F)|\leq 7\sigma.
\end{equation}
\subsection{Proof of Theorem \ref{cri}}


In this subsection, we shall give a proof of Theorem \ref{cri}. Let $H$ be an $F$-free graph on $n$ vertices, and $\mathbf{x}=(x_{1},\ldots,x_{n})$ be a unit eigenvector corresponding to $q(H)$, and let $x=\min\{x_{1},\ldots,x_{n}\}$.

\begin{lem}[See \cite{AN2013}]\label{min}
Let $G$ be a graph of order $n$ with $q(G)=q$ and minimum degree $\delta(G)=\delta$. If  $\mathbf{x}=(x_{1},\ldots,x_{n})$ is a nonnegative unit  eigenvector to $q$, then the value $x=\min\{x_{1},\ldots,x_{n}\}$ satisfies 
$$x^{2}(q^{2}-2q\delta+n\delta)\leq \delta.$$
\end{lem}

We first present two lemmas under the conditions of Theorem \ref{cri}.

\begin{lem}\label{mind}
Suppose that $q(H)\geq q(\mathcal{G}_{n})$  and
$\delta(H)\leq(\pi(F)-\varepsilon)n.$
Then for sufficiently large $n$, 
$$x^{2}<\frac{1-\varepsilon}{n}.$$
\end{lem}
\begin{proof}
We denote $q=q(H)$ and $\delta(H)=\delta$ for short.
Using Lemma \ref{min}, and the fact that
$$q^{2}-2q\delta+n\delta=(q-\delta)^{2}+\delta(n-\delta)>0,$$
we get
$$x^{2}\leq \frac{\delta}{q^{2}-2q\delta+n\delta}.$$
Noting that the right-hand side increases with $\delta$ and decreases with $q$ on $[\delta,+\infty)$. Setting $\varepsilon'=\pi(F)\varepsilon/(\pi(F)+\varepsilon)$
, in view of (\ref{qn}), for
sufficiently large $n$ we have
$$q\geq q(\mathcal{G}_{n})\geq (2\pi(F)-\varepsilon')n>(\pi(F)-\varepsilon)n\geq \delta.$$
Combining these with $0<\varepsilon<\frac{1}{2}$ and $\delta=\delta(H)\leq(\pi(F)-\varepsilon)n$, we obtain that
\begin{displaymath}
\begin{split}
x^{2}n &\leq \frac{\delta n}{q^{2}-2q\delta+n\delta}
=\frac{\delta n}{n\delta+(q-\delta)^{2}-\delta^{2}}\\ 
&\leq \frac{\pi(F)-\varepsilon}{\pi(F)-\varepsilon
+(\pi(F)+\varepsilon-\varepsilon')^{2}-(\pi(F)-\varepsilon)^{2}}\\
&\leq \frac{\pi(F)-\varepsilon}{\pi(F)-\varepsilon
+4\pi(F)\varepsilon-2(\pi(F)+\varepsilon)\varepsilon'}\\
&= \frac{\pi(F)-\varepsilon}{\pi(F)-\varepsilon
+2\pi(F)\varepsilon}\\
&= 1- \frac{2\pi(F)\varepsilon}{\pi(F)-\varepsilon
+2\pi(F)\varepsilon}\\ 
&\leq  1-\varepsilon, 
\end{split}
\end{displaymath}
completing the proof. 
\end{proof}

\begin{lem}\label{dv}
Let $u$ be a vertex for which $x_{u}=x$. Suppose that $q(H)\geq q(\mathcal{G}_{n})$  and $x^{2}<(1-\varepsilon)/n$. Then for sufficiently large $n$, we have
$$q(H-u)\geq q(H)
\left( 1- \frac{1-\varepsilon/6}{n-1} \right),$$
and
$$q(H-u)> q(\mathcal{G}_{n-1}).$$
\end{lem}
\begin{proof}
Firstly, the Rayleigh principle implies that
\begin{align}
q(H) &=\mathbf{x}^\mathrm{T}Q(H)\mathbf{x}=\sum_{ij\in E(H)}(x_{i}+x_{j})^{2}=\sum_{ij\in E(H-u)}(x_{i}+x_{j})^{2}+\sum_{j\in N (u) }(x+x_{j})^{2} \notag \\
&= \sum_{ij\in E(H-u)}(x_{i}+x_{j})^{2}+d(u)x^{2}+2x\sum_{j\in N (u) }x_{j}+\sum_{j\in N (u) }x_{j}^{2} \notag \\
&\leq (1-x^{2})q(H-u)+d(u)x^{2}+2x\sum_{j\in N (u) }x_{j}+\sum_{j\in N (u) }x_{j}^{2}.
\label{e1}  
\end{align}
From the eigenequation  for $q(H)$ at the vertex $u$, we have 
$$(q(H)-d(u))x=\sum_{j\in N (u) }x_{j}.$$
Then we see that
\begin{displaymath}
\begin{split}
d(u)x^{2}+2x\sum_{j\in N (u) }x_{j}+\sum_{j\in N (u) }x_{j}^{2}
&= d(u)x^{2}+2(q(H)-d(u))x^{2}+\sum_{j\in N (u) }x_{j}^{2}\\
&\leq d(u)x^{2}+2(q(H)-d(u))x^{2}+1-(n-d(u))x^{2}\\
&= 2q(H)x^{2}-nx^{2}+1.
\end{split}
\end{displaymath}
Combining this with (\ref{e1}), we find that
\begin{equation}\label{e2}
q(H-u)\geq q(H)\frac{1-2x^{2}}{1-x^{2}}-\frac{1-nx^{2}}{1-x^{2}}.
\end{equation} 
 Note that $\pi(F)=1-1/r \geq2/3$ as $r\geq 3$. 
We get from (\ref{qn}) that for enough large $n$, 
$$q(H)\geq q(\mathcal{G}_{n})\geq(2\pi(F)-1/12)n\geq5n/4.$$
Together with (\ref{e2}) and $x^{2}<(1-\varepsilon)/n$, we get
\begin{align*}
\frac{q(H-u)}{n-2} 
&\geq \frac{q(H)}{n-1}\bigg(1+\frac{1}{n-2}\bigg)\frac{1-2x^{2}}{1-x^{2}}-\frac{1-nx^{2}}{(n-2)(1-x^{2})}\\
&= \frac{q(H)}{n-1}\bigg(1+\frac{1-nx^{2}}{(n-2)(1-x^{2})}\bigg)-\frac{1-nx^{2}}{(n-2)(1-x^{2})}\\
&\geq \frac{q(H)}{n-1}\bigg(1+\frac{1-nx^{2}}{5(n-2)(1-x^{2})}\bigg)+\frac{q(H)}{n-1}\cdot \frac{4(1-nx^{2})}{5(n-2)(1-x^{2})}-\frac{1-nx^{2}}{(n-2)(1-x^{2})}\\
&\geq \frac{q(H)}{n-1}\bigg(1+\frac{1-nx^{2}}{5(n-1)}\bigg)\\
&\geq \frac{q(H)}{n-1}\bigg(1+\frac{\varepsilon}{5(n-1)}\bigg). 
\end{align*} 
Hence, it follows that 
\begin{equation*}
q(H-u)\geq q(H)\bigg(1-\frac{1}{n-1}\bigg)\bigg(1+\frac{\varepsilon}{5(n-1)}\bigg)\geq q(H) 
\left(1- \frac{1-\varepsilon/6}{n-1} \right). 
\end{equation*} 
On the other hand, we have 
\begin{displaymath}
\begin{split}
\frac{q(H-u)}{n-2} 
&\geq \frac{q(H)}{n-1}\bigg(1+\frac{1-nx^{2}}{(n-2)(1-x^{2})}\bigg)-\frac{1-nx^{2}}{(n-2)(1-x^{2})}\\
&\geq \frac{q(H)}{n-1}+\bigg(\frac{5}{4}-1\bigg)\times \frac{(1-nx^{2})}{(n-2)(1-x^{2})}\\
&\geq \frac{q(H)}{n-1}+\frac{\varepsilon}{4(n-2)}.
\end{split} 
\end{displaymath}
For sufficiently large $n$, by (\ref{qn}) and (\ref{beg}), we have $q(\mathcal{G}_{n-1})\leq(2\pi(F)+\sigma)(n-1)$ and $q(H)\geq q(\mathcal{G}_{n})\geq q(\mathcal{G}_{n-1})+2\pi(F)-7\sigma$. Thus, 
\begin{displaymath}
\begin{split}
q(H-u)&\geq q(H)\bigg(1-\frac{1}{n-1}\bigg)+\frac{\varepsilon}{4}\\
&\geq (q(\mathcal{G}_{n-1})+2\pi(F)-7\sigma)\bigg(1-\frac{1}{n-1}\bigg)+\frac{\varepsilon}{4}\\
&\geq q(\mathcal{G}_{n-1})+2\pi(F)-7\sigma-(2\pi(F)+\sigma)-\sigma+\frac{\varepsilon}{4}\\
&= q(\mathcal{G}_{n-1})-9\sigma+\frac{\varepsilon}{4}\\
&> q(\mathcal{G}_{n-1}),
\end{split}
\end{displaymath}
where the last inequality follows from  $\sigma<\varepsilon/36$. 
\end{proof}

\noindent {\bf Fact 1.} If  $0<x< \frac{1}{2}$ and $0<a<1$, then $\ln(1-ax)+ax+x^{2}>0$.

\begin{proof}[Proof of Fact 1]
We denote $f(x):=\ln(1-ax)+ax+x^{2}$. The derivative of $f(x)$ is given as 
 $$f'(x)=-\frac{a}{1-ax}+a+2x=\frac{x(2-a^{2}-2ax)}{1-ax}.$$
Given $0<x<1/2$ and $0<a<1$, we have $f'(x)>0$, which implies that $f(x)$ is strictly increasing for $x\in (0,{1}/{2})$. Since $f(0)=0$, it follows that $f(x)>0$ for all $0<x<1/2$. 
\end{proof}

\noindent {\bf Fact 2.} If  $x>1$, then $\frac{1}{x}<\ln x-\ln(x-1)$ and $\frac{1}{x^{2}}<\frac{1}{x-1}-\frac{1}{x}$.

\begin{proof}[Proof of Fact 2]
First, we observe that $\frac{1}{x}<\ln x-\ln(x-1)$ is equivalent to 
$\ln(1-\frac{1}{x})+\frac{1}{x}<0$ for $x>1$. It is sufficient to show
that $g(t):=\ln(1-t)+t<0$, where $0<t<1$. Since $g'(t)=-\frac{t}{1-t}<0$, we know that $g(t)$ is strictly decreasing. As $g(0)=0$, we get $g(t)<0$. 
The second inequality is straightforward, since 
$\frac{1}{x^{2}}<\frac{1}{x(x-1)}=\frac{1}{x-1}-\frac{1}{x}$ holds for $x>1$. 
\end{proof}

\medskip{Next, we complete the proof Theorem \ref{cri}.}

\begin{proof}[{\bf Proof of Theorem \ref{cri}}]
Let $H$ be an $n$-vertex $F$-free graph
 with $q(H)\geq q(\mathcal{G}_{n})$. 
Our goal is to show that $H\in \mathcal{G}_n$. 
 We may assume that $N$ is sufficiently large to apply Lemmas
 \ref{mind} and \ref{dv} for $n\geq N$, and in addition $N$ is large such that  for $n\geq N$, by (\ref{qn}),
 we have $q(\mathcal{G}_{n})\geq2(1-\varepsilon)\pi(F)n$. Let $n_{0}=(Ne/(1-\varepsilon)\pi(F))^{6/\varepsilon}$. Clearly, we have $n_{0}>N$.  We can construct a sequence of graphs $H=H_{n}$, $H_{n-1},\ldots$, 
 where $n\geq n_{0}$, and $H_{i}$ is a $i$-vertex  $F$-free graph with $q(H_{i})>q(\mathcal{G}_{i})$ for $i<n$.
 To do so, suppose that $\delta(H_{i})\leq(\pi(F)-\varepsilon)i$ and $\mathbf{x}=(x_{1},\ldots,x_{i})$ is a unit eigenvector to $q(H_{i})$, where $i\leq n$. Let $u\in V(H_{i})$ and $x_{u}=\min\{x_{1},\ldots,x_{i}\}$.  Then for $i\geq N$, Lemma
 \ref{mind} implies that  $x_{u}^{2}<(1-\varepsilon)/i$. We set $H_{i-1}=H_{i}-u$. By
 Lemma \ref{dv}, we have
 $$q(H_{i-1})\geq q(H_{i})(1-(1-\varepsilon/6)(i-1)^{-1}),$$
and
$$q(H_{i-1})> q(\mathcal{G}_{i-1}).$$

We claim that this process terminates at some $H_{t}$ with $t>N$. Suppose towards a contradiction that the sequence of graphs reaches $H_{N}$. Then we have
\begin{align*}
q(H_{N})
&\geq q(H_{N+1})(1-(1-\varepsilon/6)N^{-1})\\
&\geq q(H_{n})\prod_{i=N+1}^{n}(1-(1-\varepsilon/6)(i-1)^{-1})\\
&\geq q(H_{n})\exp\bigg(-\sum_{i=N+1}^{n}\big((1-\varepsilon/6)(i-1)^{-1}+(i-1)^{-2}\big)\bigg)\\
&\geq q(H_{n})\exp\big(-(1-\varepsilon/6)\ln(n/(N-1))-1\big)\\
&\geq 2(1-\varepsilon)\pi(F)n\big(\frac{n}{N-1}\big)^{-(1-\varepsilon/6)}e^{-1}\\
&\geq  2(1-\varepsilon)\pi(F)n^{\varepsilon/6}e^{-1}\\
&\geq  2(1-\varepsilon)\pi(F)e^{-1}n_{0}^{\varepsilon/6}\\
&\geq  2N,
\end{align*}
where the third inequality follows from Fact $1$, and the fourth inequality follows from Fact $2$. It is impossible that the signless Laplacian spectral radius of any $N$-vertex graph is  at least $2N$. So this contradiction shows that the process terminates at some $H_{t}$ with $t>N$. By contradiction, it follows that 
$\delta(H_{t})>(\pi(F)-\varepsilon)t,$
so we have $H_{t}\in \mathcal{G}_{t}$, and hence $q(H_{t})\leq q(\mathcal{G}_{t})$. If $t<n$, then by Lemma \ref{dv},  we get $q(H_{t})>q(\mathcal{G}_{t})$, a contradiction. Therefore $t=n$, which implies that  $$\delta(H)=\delta(H_{n})>(\pi(F)-\varepsilon)n.$$
By the definition of $\mathcal{G}_{n}$, we know that  
$H\in \mathcal{G}_{n}$.
\end{proof}

\section{Proofs of Theorems \ref{color} and \ref{thm-1-6}}

A famous result of Andr\'{a}sfai, Erd\H{o}s and S\'{o}s \cite{AES1974} states that if $r\ge 2$ and $G$ is an $n$-vertex $K_{r+1}$-free graph with minimum degree $\delta (G) > \frac{3r-4}{3r-1}n$, then $G$ must be $r$-partite. 
Moreover, Erd\H{o}s and Simonovits \cite{ES1973} proved the following extension for color-critical graphs. 

\begin{lem}[See \cite{ES1973}] 
\label{lem-ES1973}
Let $F$ be a color-critical graph 
with $\chi (F)=r+1\ge 3$.  There is $n_0$ such that 
if $G$ is an $F$-free graph on $n\ge n_0$ vertices with 
 $\delta (G) >  \frac{3r-4}{3r-1}n$, 
 then $G$ is $r$-partite. 
\end{lem}

Now, we are ready to prove Theorem \ref{color}. 

\begin{proof}[{\bf Proof of Theorem \ref{color}}]  
Assume that $F$ is a color-critical graph with 
$\chi (F)=r+1\ge 4$. We will apply Theorem \ref{cri}. 
Let $\varepsilon>0$ and $\sigma < \varepsilon/36$ be sufficiently small real numbers. 
Theorem \ref{thm-Sim} implies that
$\mathrm{ex}(n,F)=e(T_{n,r})=\frac{r-1}{2r}n^{2}+O(1)$ and $\pi(F)=\frac{r-1}{r}$.  
Then for sufficiently large $n$, we have 
\begin{displaymath}
\begin{split}
\bigg|\mathrm{ex}(n,F)-\mathrm{ex}(n-1,F)-\pi(F)n\bigg|
=\bigg|\frac{r-1}{2r}n^{2}-\frac{r-1}{2r}(n-1)^{2}-\frac{r-1}{r}n+O(1)\bigg| \leq \sigma n. 
\end{split}
\end{displaymath} 
So the condition (\ref{dl1}) of Theorem \ref{cri} is satisfied. 
Let $\mathcal{G}_n$ be the collection  of all $n$-vertex $F$-free graphs $G$ with minimum degree $\delta (G) > (\frac{r-1}{r} - \varepsilon) n$. Since $\varepsilon >0$ is sufficiently small so that $\frac{r-1}{r} -\varepsilon > \frac{3r-4}{3r-1}$. 
By Lemma \ref{lem-ES1973}, we know that for sufficiently large $n$, each member of $\mathcal{G}_n$ is an $r$-partite graph on $n$ vertices. 
Observe that $T_{n,r}$ is $F$-free and $\delta (T_{n,r})=\lfloor \frac{r-1}{r}n \rfloor $. It follows that $T_{n,r}\in \mathcal{G}_n$.  
 By Theorem \ref{partite}, we have $q(\mathcal{G}_n) =q(T_{n,r})= \frac{2(r-1)}{r}n + o(1)$. 
 Then for sufficiently large $n$, we get  
\begin{displaymath}
\begin{split}
\bigg|q(\mathcal{G}_{n})-4\mathrm{ex}(n,F)n^{-1}\bigg|
=\bigg|q(T_{n,r})-4e(T_{n,r})n^{-1}\bigg|
= o(1)\leq\sigma.
\end{split}
\end{displaymath}
Thus, the condition (\ref{dl2}) of Theorem \ref{cri} is met. Therefore,  Theorem \ref{cri} implies  that for sufficiently large $n$, we have $q(G)\le q(\mathcal{G}_n)= q(T_{n,r})$, and the equality holds if and only if $G=T_{n,r}$. 
\end{proof}

Finally, we are ready to show Theorem \ref{thm-1-6}. 
In what follows, we provide two algebraic proofs of Theorem \ref{thm-1-6}. This is quite different from the combinatorial proof of (\ref{thm-NR-degree}). 

\begin{proof}[{\bf Proof of Theorem \ref{thm-1-6}}]
In our argument, we need to use a lower bound on $q(G)$.
This bound can be found in \cite[Theorem 2.1]{LL2009-Zagreb} or \cite[Lemma 3]{Zhou2010}, which  states that
\begin{equation} \label{eq-signless}
  q (G) \ge \frac{1}{m} \sum_{v\in V(G)} d^2(v),
  \end{equation}
with equality if and only if  $d(u)+d(v)$ are equal for
any $uv \in E(G)$. 
Since $q(T_{n,r}) \le (1- \frac{1}{r})2n$, 
 Theorem \ref{color} implies that if
$G$ is an $F$-free graph on $n$ vertices, then
\begin{equation} \label{eq-q-le}
q(G) \le \left( 1-\frac{1}{r} \right)2n. 
\end{equation}
Combining (\ref{eq-signless}) with (\ref{eq-q-le}),
we obtain
$\sum_{v\in V(G)} d^2(v) \le 2\left(1-\frac{1}{r} \right)mn$, 
as needed. 
\end{proof}

At the end of this paper,
we give another algebraic proof of Theorem \ref{thm-1-6}.

\begin{proof}[{\bf Second proof of Theorem \ref{thm-1-6}}]
The well-known Hofmeistar inequality shows that
\begin{equation*}
  \lambda^2 (G) \ge \frac{1}{n} \sum_{v\in V} d^2(v),
  \end{equation*}
with equality if and only if $G$ is either regular or bipartite semi-regular. 
We may assume that $G$ is the graph achieving the maximum degree power. 
Then $G$ is connected and $n-1 \le m$. 
So $m$ is also sufficiently large. It was shown by Li, Liu and Zhang  \cite{LLZ2025-stability} that if
$F$ is a color-critical graph with $\chi (F)=r+1\ge 4$ and $m$ is sufficiently large,  then for every $m$-edge $F$-free graph $G$, we have 
\begin{equation*}
  \lambda^2(G) \le \left( 1-\frac{1}{r} \right)2m.
  \end{equation*}
Therefore,
it follows that $\sum_{v\in V(G)} d^2(v) \le 2\left(1-\frac{1}{r} \right)mn$, as expected.
\end{proof}

\section{Concluding remarks}
In this paper, we have studied the extremal graph problem for color-critical graphs. 
Firstly, extending a result of He, Jin and Zhang \cite{HJZ2013}, we proved that the $r$-partite Tur\'{a}n graph $T_{n,r}$ is the unique spectral extremal graph that achieves the maximum signless Laplacian spectral radius when we forbid a color-critical graph $F$ with chromatic number $\chi (F)=r+1\ge 4$;  see Theorem \ref{color}.  
This spectral result also extends a classical theorem of Simonovits \cite{S1968}. 
 As an application, we determined the upper bound on the degree power in an $F$-free graph $G$ by showing that $\sum_{v\in V(G)} d^2(v) \le 2(1-\frac{1}{r})mn$; see Theorem \ref{thm-1-6}.  
 This extends a result of Nikiforov and Rousseau \cite{NR2004} and recovers a result of Bollob\'{a}s and Nikiforov \cite{BK2012}. 
 At the end of this paper, we conclude some extremal graph problems concerning the signless Laplacian spectral radius for interested readers.

\subsection{Forbidding complete bipartite graphs plus an edge}

We mention that our result in Theorem \ref{color} holds for any color-critical graph $F$ with chromatic number $\chi (F)\ge 4$. However, 
when we consider the graph $F$ with $\chi (F)=3$, the extremal problem involving the signless Laplacian spectral radius lies beyond the range of our investigation. 
For $2\le s \le t$, 
we denote by $K_{s,t}^+$ the graph obtained from the complete bipartite graph $K_{s,t}$ by embedding an edge into the partite set of size $s$. It is easy to see that $K_{s,t}^+$ is a color-critical graph with $\chi (K_{s,t}^+)=3$. 
Note that every color-critical graph $F$ with $\chi (F)=3$ is a subgraph of $K_{s,t}^+$ for some integers $s,t\ge 2$. 
 Therefore, $K_{s,t}^+$ serves as a natural candidate for our study of color-critical graphs with chromatic number three. 
In particular, Theorem \ref{thm-Sim} implies that if $n$ is sufficiently large and $G$ is an $n$-vertex $K_{s,t}^+$-free graph, then $e(G)\le e(T_{n,2})$, where the equality holds if and only if $G=T_{n,2}$. 
Nevertheless, the spectral extremal problem for $K_{s,t}^+$ is significantly different when we consider the signless Laplacian spectral radius. 
In this case, we propose the following conjecture.  
Let $2\le s \le t$ be positive integers. 
We define the following families: 

\begin{itemize}
\item 
Let $\mathcal{L}_{n,s,t}$ be the family of graphs that are the join of a clique $K_{s-1}$ and a (nearly) $(t-1)$-regular triangle-free graph of order $n-s+1$. 

\item 
Let $\mathcal{Y}_{n,t}$ be the family of graphs that are the join of an independent set $I_{t-1}$ and a  (nearly) $(t-1)$-regular triangle-free graph of order $n-t+1$. 
\end{itemize} 
Clearly, all graphs in both $\mathcal{L}_{n,s,t}$ and $\mathcal{Y}_{n,t}$ are $K_{s,t}^+$-free. 
We point out that 
such a regular graph in the above may not exist, but we can always choose a nearly regular graph. 
It is well-known that 
if $t$ is even and $n$ is odd, then there exist nearly $(t-1)$-regular graphs of order $n$ whose degree sequence is $(t-1,\ldots ,t-1,t-2)$. 
Otherwise, there exist $(t-1)$-regular graphs of order $n$.

\begin{conj}
Let $2\le s\le t$ and $n\ge s+t$.  
If $G$ is an $n$-vertex $K_{s,t}^+$-free graph with the maximal signless Laplacian spectral radius, then 
$G$ is a member of either $\mathcal{L}_{n,s,t}$ or $\mathcal{Y}_{n,t}$.   
\end{conj}

In the case $s=2$, the graph $K_{s,t}^+$ reduces to the classical book graph. 
It is easy to verify that  $q(\mathcal{L}_{n,2,t}) < q(\mathcal{Y}_{n,t})$ for large $n$, 
where $q(\mathcal{L}_{n,2,t}):=\max\{q(G):G\in \mathcal{L}_{n,2,t}\}$. 
This case was recently studied by Chen, Jin and Zhang \cite{CJZ2025}, who proved that the $Q$-spectral extremal graphs\footnote{ For the sake of simplicity, we say that a graph $G$ is a {\it $Q$-spectral extremal graph} of $F$ if $G$ is an $F$-free graph and $G$ achieves the maximal signless Laplacian spectral radius.} of $K_{2,t}^+$ lie in $\mathcal{Y}_{n,t}$.  
We mention that the $Q$-spectral extremal graphs are quite different when we forbid  $K_{2,t}$ or $K_{2,t}^+$ as a subgraph. 
It was shown in \cite{AFNP2016} that the $Q$-spectral extremal graphs of $K_{2,t}$ are in $\mathcal{L}_{n,2,t}$.

In the case $s \ge 3$, the phenomenon may be different from the case $s=2$. 
For fixed $s\ge 3$, if $n$ is sufficiently large, then we 
can compute that  $q(\mathcal{L}_{n,s,t}) > q(\mathcal{Y}_{n,t})$.  
Based on this observation, we venture to speculate that for every $3\le s\le t$ and sufficiently large $n$, the $Q$-spectral extremal graphs are samely located in $\mathcal{L}_{n,s,t}$ when excluding either $K_{s,t}$ or $K_{s,t}^+$ as a subgraph.

\subsection{Forbidding color-$k$-critical graphs} 

Recall that an {\it induced matching} of size $k$ in a graph $F$ is a set of $k$ edges such that no two of them intersect in a common vertex or are joined by an edge of $F$; that is, an induced matching is a matching that forms an induced subgraph. 
For an integer $k\ge 1$, a graph $F$ is called {\it color-$k$-critical} if 
there exists an induced matching\footnote{We would like to emphasize that the $k$ edges are required to form an induced matching in  $F$. However, most references in the literature have overlooked this requirement.} of size $k$ whose deletion decreases its chromatic number, and deleting any $k-1$ vertices of $F$ does not decrease its chromatic number.  
It is worth highlighting that this family includes a wide range of graphs, such as color-critical graphs, the Petersen graph, Kneser graphs $KG(n,2)$, and disjoint unions of cliques. 

We denote $H_{n,r,k}:=K_{k-1}\vee T_{n-k+1,r}$, 
which is the graph obtained by joining each vertex of the complete graph $K_{k-1}$ to each vertex of the $r$-partite Tur\'{a}n graph $T_{n-k+1,r}$. 
In 1974, Simonovits \cite{Sim1974} determined the unique extremal graph for color-$k$-critical graphs.

\begin{theorem}[Simonovits \cite{Sim1974}]
    Let $k\ge 1,r\ge 2$ and $F$ be a color-$k$-critical graph with $\chi (F)=r+1$. If $n$ is sufficiently large and $G$ is an $n$-vertex $F$-free graph, then $e(G)\le e(H_{n,r,k})$, with equality if and only if $G=H_{n,r,k}$. 
\end{theorem}

Motivated by this result, we propose the following conjecture. 

\begin{conj} \label{conj-2}
    Let $k\ge 2$ and $F$ be a color-$k$-critical graph with $\chi (F)= r+1\ge 4$. If $n$ is sufficiently large and $G$ is an $n$-vertex $F$-free graph, then 
    $q(G)\le q(H_{n,r,k})$, 
    with equality if and only if $G=H_{n,r,k}$.  
\end{conj}

The case $k=1$ in Conjecture \ref{conj-2} reduces to Theorem \ref{color}. 
In the case $k\ge 2$, 
we would like to point out that our proof of Theorem \ref{color} cannot be extended to the setting of Conjecture \ref{conj-2},  since the $F$-free graphs do not have the degree-stable property. There is no related result corresponding to Lemma \ref{lem-ES1973} whenever $F$ is a color-$k$-critical graph with $k\ge 2$. 
In particular, 
the case $F=2K_3$ was recently studied by Zhang and Wang \cite{ZW2024}.

\begin{thebibliography}{99}

\bibitem{AN2013} N.M.M. de Abreu, V. Nikiforov, Maxima of the $Q$-index: Graphs with bounded clique number,
 {Electronic J. Linear Algebra}, 24 (2013) 121--130.
 
 \bibitem{AFNP2016}
M. Aguieiras, A. de Freitas,  V. Nikiforov, L. Patuzzi,
Maxima of the $Q$-index: Graphs with no $K_{s,t}$,
Linear Algebra  Appl. 496 (2016) 381--391.

\bibitem{AES1974}
B. Andr\'{a}sfai, P. Erd\H{o}s,V.T. S\'{o}s, 
On the connection between chromatic number, 
maximal clique and minimum degree of a graph, 
Discrete Math. 8 (1974) 205--218. 

\bibitem{Bollobas78}
B. Bollob\'as, Extremal Graph Theory, Academic Press, New York, 1978.



\bibitem{BT1981}
B. Bollob\'as, A. Thomason,
Dense neighbourhoods and Tur\'{a}n's theorem,
J. Combin. Theory Ser. B 31 (1981) 111--114.

\bibitem{BK2004}
B. Bollob\'{a}s, V. Nikiforov,
Degree powers in graphs with forbidden subgraphs,
 Electron.  J. Combin. 11 (2004), R42.

\bibitem{BK2012}
B. Bollob\'{a}s, V. Nikiforov,
Degree powers in graphs: the Erd\H{o}s--Stone theorem,
Combin. Probab. Comput. 21 (2012) 89--105.

\bibitem{Bon1983}
J.A. Bondy,
Large dense neighbourhoods and Tur\'{a}n's theorem,
J. Combin. Theory Ser. B 34 (1983) 109--111.



\bibitem{BH2012}
A.E. Brouwer, W.H. Haemers, 
Spectra of Graphs, UTX, Springer, 2012. 

\bibitem{CWZ2022}
W. Chen, B. Wang, M. Zhai, 
Signless Laplacian spectral radius of graphs without short cycles or long cycles, 
Linear Algebra Appl. 645 (2022) 123--136. 

\bibitem{CLZ2020} 
M.-Z, Chen, A.-M. Liu, X.-D. Zhang, The signless Laplacian spectral radius of graphs with forbidding linear forests,  
{Linear Algebra Appl.}, 591 (2020) 25--43.

\bibitem{CZ2021}
M.Z. Chen, X.D. Zhang, 
On the signless Laplacian spectral radius of $K_{s,t}$-minor free graphs, 
Linear Multilinear Algebra 69 (10) (2021) 1922--1934.

  \bibitem{CLZ2024}
M.-Z. Chen, A.-M. Liu, X.-D. Zhang,
The signless Laplacian spectral radius of graphs without
intersecting odd cycles,
Electron. J. Linear Algebra 40 (2024) 370--381.

\bibitem{CJZ2025}
M.-Z. Chen, Y.-L. Jin, P.-L. Zhang, 
The signless Laplacian spectral radius of book-free graphs, 
Discrete Math. 348 (2025), No. 114448. 

\bibitem{FNP2013} 
M.A.A. de Freitas, V. Nikiforov, L. Patuzzi, Maxima of the $Q$-index: forbidden $4$-cycle and $5$-cycle,   {Electron. J. Linear Algebra.}, 26 (2013) 905--916.

\bibitem{Dirac1963}
G.A. Dirac, 
Extension of Tur\'{a}n's theorem on graphs, 
Acta Math. Acad. Sci. Hungar 14 (1963) 417--422.






\bibitem{ES1946} 
P. Erd\H{o}s, A.H. Stone, On the structure of linear graphs,  {Bull. Amer. Math. Soc.}, 52 (1946) 1087--1091.



\bibitem{ES1966} P. Erd\H{o}s, M. Simonovits, A limit theorem in graph theory,
{Stud. Sci. Math. Hungar.}, 1 (1966) 51--57.

\bibitem{ES1973} 
P. Erd\H{o}s, M. Simonovits, On a valence problems in extremal graph theory,  {Discrete Math.}, 5 (1973) 323--334.



 
 


 
 \bibitem{G2025}
D. Gerbner, On degree powers and counting stars in $F$-free graphs, European J. Combin. 126 (2025) 104--135.

\bibitem{GR2001}
C. Godsil, G. Royle, 
Algebraic Graph Theory, GTM 207, Springer, New York, 2001. 

\bibitem{Gui1996}
B.D. Guiduli, 
 Spectral extrema for graphs, 
 Ph.D. dissertation, 
University of Chicago, 1996. (unpublished)
See \url{http://people.cs.uchicago.edu/~laci/students/guiduli-phd.pdf} 
 
\bibitem{HJZ2013} 
B. He, Y.-L. Jin, X.-D. Zhang, Sharp bounds for the signless Laplacian spectral radius in terms of clique number,  {Linear Algebra Appl.}, 438(2013) 3851--3861. 

 \bibitem{KNS1964} 
 G. Katona, T. Nemetz, M. Simonovits, On a problem of Tur\'an in the theory of graphs,
{ Mat. Lapok}, 15 (1964) 228--328.

\bibitem{KLM2014} P. Keevash, J. Lenz, D. Mubayi, Spectral extremal problems for hypergraphs, 
 SIAM J. Discrete Math., 28 (4) (2014) 1838--1854. 

\bibitem{KW2017}
 Q. Kong, L. Wang, Upper bounds on the $Q$-spectral radius of book-free and/or $K_{s,t}$-free graphs, Electron. J. Linear Algebra 32 (2017) 447--453. 

\bibitem{LLQ2019} 
Y. Lan, H. Liu, Z. Qin, Y. Shi, Degree powers in graphs with a forbidden forest, Discrete Math. 342 (2019) 821--835.

\bibitem{LLPLMA18}
Y. Li,  Y. Liu,  X. Peng,  
Signless Laplacian spectral radius and Hamiltonicity of graphs with large minimum degree, 
Linear Multilinear Algebra 66 (2018) 2011--2023.

\bibitem{LiLF} 
Y. Li, W. Liu, L. Feng, A survey on spectral conditions for some extremal graph problems,  {Adv. Math. (China)}, 51 (2022) 193--258.
 

\bibitem{LLZ2024-book-4-cycle}
Y. Li, H. Liu, S. Zhang,
More on Nosal's spectral theorem: Books and 4-cycles,
23 pages, (2024), \url{https://www.ibs.re.kr/ecopro/wp-content/uploads/2024/08/Nosal.pdf}.

\bibitem{LLZ2025-stability}
Y. Li, H. Liu, S. Zhang, An edge-spectral stability theorem with applications on the Brualdi--Hoffman--Tur\'{a}n problem,
(2025), in preparation.

\bibitem{LL2009-Zagreb}
M. Liu, B. Liu, New sharp upper bounds for the first Zagreb index,
MATCH Commun. Math. Comput. Chem., 62(3) (2009) 689--698.

\bibitem{LMX2022}
R. Liu, L. Miao, J. Xue, 
Maxima of the $Q$-index of non-bipartite $C_3$-free graphs, 
Linear Algebra Appl. 673 (2023) 1--13.


\bibitem{Mer1998}
R. Merris, A note on Laplacian graph eigenvalues, 
Linear Algebra Appl. 295 (1998) 33--35.

\bibitem{Mub2010}
 D. Mubayi, 
 Counting substructures I: Color critical graphs, 
 Adv. Math. 225 (5) (2010) 2731--2740.

\bibitem{NR2004}
V. Nikiforov, C.C. Rousseau, Large generalized books are $p$-good,
J. Combin. Theory Ser B 92 (2004) 85--97.

\bibitem{Niki2007laa2} 
V. Nikiforov, 
Bounds on graph eigenvalues II, 
Linear Algebra Appl. 427 (2007) 183--189. 

\bibitem{Niki2009-degree}
V. Nikiforov,
Degree powers in graphs with a forbidden even cycle,
Electron. J. Combin. 15 (2009), R107. 

\bibitem{N2009} 
V. Nikiforov, Spectral saturation: inverting the spectral Thr\'an theorem,  
Electron. J. Combin., 16 (1) (2009), R33.




\bibitem{Niki2011}
V. Nikiforov, 
Some new results in extremal graph theory, 
in Surveys in Combinatorics, 
Cambridge University Press, 2011, pp. 141--181.

\bibitem{NY1} 
V. Nikiforov, X. Yuan, 
Maxima of the $Q$-index: graphs without long path,  
{Electron. J. Linear Algebra}, 27 (2014) 504--514.

\bibitem{NY2} 
V. Nikiforov, X. Yuan, Maxima of the $Q$-index: forbidden even cycles,  {Linear Algebra Appl.}, 471 (2015) 636--653.





\bibitem{PLZ2020}
Y. Pan, J. Li,  W. Zhao, 
Signless Laplacian spectral radius and fractional matchings in graphs. Discrete Math. 343 (2020), No. 112016. 

\bibitem{PY2017}
O. Pikhurko, Z. Yilma, 
Supersaturation problem for color-critical graphs, 
J. Combin. Theory Ser. B 123 (2017) 148--185.


\bibitem{S1968} 
M. Simonovits, A method for solving extremal problems in graph theory, stability problems,  {In Theory of Graphs (Proc. Colloq., Tihany, 1966)}, pp.279-319, Academic Press, New York, 1968.

\bibitem{Sim1974}
 M. Simonovits, Extremal graph problems with symmetrical extremal graphs, additionnal chromatic conditions, 
 Discrete Math. 7 (1974) 349--376.

\bibitem{Turan41}
P. Tur\'{a}n,
On an extremal problem in graph theory,
Mat. Fiz. Lapok 48 (1941), pp. 436--452.
(in Hungarian).

\bibitem{WZ2023} 
B. Wang, M. Zhai, Maxima of the $Q$-index: Forbidden a fan,  {Discrete Math.}, 34 (2023), No.113264.

\bibitem{Wil1986} 
H. Wilf, 
Spectral bounds for the clique and indendence numbers of graphs, 
J. Combin. Theory Ser. B 40 (1986) 113--117.  

\bibitem{Yu2008}
G. Yu, On the maximal signless Laplacian spectral radius of graphs with given matching number, Proc. Jpn. Acad., Ser. A, Math. Sci. 84 (2008) 163--166.

 \bibitem{Y2014} 
 X. Yuan, Maxima of the $Q$-index: forbidden odd cycles,  {Linear Algebra Appl.}, 458 (2014) 207--216.
 
 \bibitem{ZXL2020}
 M. Zhai, J. Xue, Z. Lou, 
 The signless Laplacian spectral radius of graphs with a prescribed number of edges, 
Linear Algebra Appl. 603 (2020) 154--165. 

 


\bibitem{ZW2024}
Y. Zhang, L. Wang, 
The signless Laplacian spectral radius of $2K_3$-free graphs, 
Discrete Math., 347 (2024), No.114075.

\bibitem{ZHG2021} 
Y. Zhao, X. Huang, H. Guo, The signless Laplacian spectral radius of graphs with no intersecting triangles,  {Linear Algebra Appl.}, 618 (2021) 12--21.


\bibitem{Zhou2010}
B. Zhou,
Signless Laplacian spectral radius and Hamiltonicity,
Linear Algebra Appl. 432 (2010) 566--570.

\end{thebibliography}
\end{document}